\documentclass{amsart}
\usepackage{amsfonts,amssymb,amscd,amsmath,enumerate,verbatim,calc}
\usepackage{xy}

\newcommand{\CM}{Cohen-Macaulay}

\newcommand{\n}{\mathfrak{n} }
\newcommand{\m}{\mathfrak{m} }

\newcommand{\rt}{\rightarrow}

\newcommand{\height}{\operatorname{height}}

\newcommand{\Syz}{\operatorname{Syz}}
\newcommand{\Tor}{\operatorname{Tor}}

\newcommand{\Supp}{\operatorname{Support}}

\newcommand{\Ext}{\operatorname{Ext}}

\theoremstyle{plain}

\newtheorem{theorem}{Theorem}[section]

\theoremstyle{definition}

\newtheorem{s}[theorem]{}

\theoremstyle{remark}

\begin{document}

\title{localization of complete intersections}
\author{Tony~J.~Puthenpurakal}
\date{\today}
\address{Department of Mathematics, IIT Bombay, Powai, Mumbai 400 076}

\email{tputhen@math.iitb.ac.in}
\subjclass{Primary 13H10; Secondary 13D02}
\keywords{complete intersections, cohomology operators}
 \begin{abstract}
Let $(A,\m)$ be an abstract complete intersection and let $P$ be a prime ideal of $A$. In \cite{Av-1} Avramov proved that $A_P$ is a complete intersection. In this paper we give an elementary proof of this result.
\end{abstract}
 \maketitle
\section{introduction}
We say a local ring $A$ is \emph{geometric complete intersection}  if $A = Q/(f_1,\ldots, f_c)$ where $Q$ is regular local and $f_1,\ldots, f_c$ is a $Q$-regular sequence. We say $A$ is an \emph{abstract complete intersection} if the completion $\widehat{A}$ is a geometric complete intersection.    Geometric complete intersections are abstract complete intersections but the converse is \emph{not} true, see  \cite[section 2]{HJ}.

If $A$ is a geometric complete intersection and $P$ is a prime ideal in $A$  then it is elementary to show that $A_P$ is a geometric complete intersection. However the corresponding result for abstract complete intersection is difficult to show. In \cite{Av-1}, Avramov showed that if $(A,\m) \rt (B, \n)$ is a flat local map of Noetherian local rings then $B$ is a complete intersection if and only if $A$ and $B/\m B$ are complete intersections (also see \cite{Av-2}). Using this fact  it is not difficult to show that if $A$ is an abstract  complete intersection then so is $A_P$ for any prime ideal $P$ in $A$; see \cite[2.3.5]{BH}.

In this short paper we give an elementary proof of the following fact:
\begin{theorem}
  \label{main}
  Let $(A,\m)$ be an abstract complete intersection of codimension $c$ and let
  $P$ be a prime ideal in $A$. Then $A_P$ is  an abstract complete intersection of codimension $\leq c$
\end{theorem}

\section{preliminaries}
In this section we discuss a few preliminary results that we need. Throughout all rings are Noetherian and all modules considered are finitely generated. If $M$ is an $A$-module then $\ell_A(M)$ denotes its length as an $A$-module.

\s \label{hilbert} Let $Q$ be a  ring (not necessarily local) and let $R = \bigoplus_{n \geq 0}R_{2n}$ be finitely generated $Q$-algebra generated by $R_2$. Let
$E = \bigoplus_{n \geq 0} E_n$ be an   $R$-module with $\ell_Q(E_n)$ finite for all $n \gg 0$. Then there exists a quasi-polynomial $p_E(z)$ of period two such that $p_E(n) = \ell_Q(E_n)$ for all $n \gg 0$.

\s\label{yoneda} Let $A$ be local and let $M$ be an $A$-module. Then \\ $\Ext^*_A(M, M) = \bigoplus_{n \geq 0}\Ext^i_A(M, M)$ is an $A$-algebra (by the Yoneda product; see \cite[III.5]{M}). This algebra is usually not a finitely generated. However if $A$ is an abstract local complete intersection then there exists finitely many central elements $u_1, \ldots, u_m \in \Ext^2_A(M, M)$ such that $\Ext^*_A(M, M)$ is finitely generated $S$-module where $S = $ the $A$-subalgebra of $\Ext^*_A(M, M)$ generated by $u_1, \ldots, u_m$; see \cite[4.9]{AGP}. In particular we may assume that $\Ext^*_A(M, M)$ is a finitely generated $A[Y_1, \ldots, Y_m]$-module for some variables $Y_i$ with $\deg Y_i = 2$ for $i = 1, \ldots, m$.

\s \label{length} Let $R$ be a  ring and let $E$ be an $A$-module with $\Supp(E) \subseteq \{ \n \}$ where $\n$ is a maximal ideal of $R$.  Set $B = R_\n$. Then
$\ell_R(E) = \ell_{B}(E_\n).$


\s \label{GL} Let $(A,\m)$ be local. Set $k = A/\m$. The Poincare series of $A$ is the formal power series $ P_A(z) = \sum_{i \geq 0} \ell_A(\Tor^A_i(k,k))z^i$. It is well-known, see \cite[3.1.3]{GL} that there exists uniquely determined non-negative integers $\epsilon_i$ with
\[
P_A(z) = \prod_{i = 0}^{\infty}\frac{(1 + z^{2i + 1})^{\epsilon_{2i}}}{(1- z^{2i+2})^{\epsilon_{2i+1}}}
\]
Set $\epsilon_i(A) = \epsilon_i$.
 Furthermore the following assertions are equivalent (see \cite[3.5.1]{GL})
 \begin{enumerate}[\rm (i)]
   \item $A$ is a complete intersection.
   \item $\epsilon_2(A) = 0$.
   \item $\epsilon_3(A) = 0$.
 \end{enumerate}

\s\label{catenary} An abstract complete intersection is a Gorenstein local ring. In particular it is Cohen-Macaulay and so universally catenary.

\s \label{ext-lemm}
Let $A$ be a Gorenstein local ring and let $0 \rt M \rt F \rt N \rt 0$ be an exact sequence with $F$ free. Then
$\Ext^i_A(M, M) = \Ext^i_A(N,N)$ for all $i \gg 0$.

\section{proof of Theorem \ref{main}}
In this section we give
\begin{proof}[Proof of Theorem \ref{main}]
We may assume $d = \dim A > 0$.

Case(1): $\height(\m/P) = 1$. \\  Let $f \in \m \setminus P$. Set
$R = A_f$. Then $\n = PA_f$ is a maximal ideal in $R$. Also let $\kappa(P)$ denote the residue field of $A_P$.

Set $M = \Syz^A_{d }(A/P)$; the $d^{th}$ syzygy of $A/P$.  By \ref{yoneda} the $A$-algebra \\
$E = \Ext^*_A(M, M)$ is finitely generated as a $S = A[Y_1, \ldots, Y_m]$-module with some variables $Y_1, \ldots,Y_m$ of degree $2$. Then $E_f$ is finitely generated as a \\ $S_f = A_f[Y_1, \ldots, Y_m]$-module.

Notice for $i > 0$ we have  $\Supp(\Ext_R^i(M_f, M_f)) \subseteq \{ \n \}$. It follows from \ref{hilbert} that there exists a   quasi-polynomial $p_E(z)$ of period two such that $p_E(n) = \ell(E_n)$ for all $n \gg 0$. Furthermore  by \ref{length}  we get $\ell_R(E_n) = \ell_{A_P}((E_n)_P)$ for $n > 0$.
As $A_P$ is Gorenstein we get that $\Ext^i_{A_P}(M_P, M_P)  = \Ext^i_{A_P}(\kappa(P), \kappa(P))$ for $i \gg 0$. It follows that
$$P_{A_P}(z) = \frac{h(z)}{(1-z^2)^r} \quad \text{for some $r \geq 0$ and $h(z)\in \mathbb{Z}[z]$}. $$
So $e_3(A_P) = 0$. It follows  from \ref{GL} that $A_P$ is an abstract complete intersection.

Case (2):
 $P$ is  a prime ideal in $A$ with $ r = \height(\m/P)> 1$. \\
   $A$ is \CM \ and so  catenary. It follows that there exists
prime ideals $P_i  $ such that $P = P_r \subseteq P_{r-1} \subseteq \cdots \subseteq P_1 \subseteq P_0 = \m$ and $\height(P_i/P_{i+1}) = 1$ for all $i$. By
Case (1)
$A_{P_1}$ is a complete intersection. As $\height_{A_{P_1}}(P_1A_{P_1}/P_2A_{A_{P_1}}) = 1$ it follows  again by Case (1) that
$$ A_{P_2} = (A_{P_1})_{P_2A_{P_1}}$$ is a complete intersection. Iterating we get that $A_P$ is an abstract complete intersection.

By considering a minimal resolution of $A/P$ over $A$ it follows easily that complexity of $\kappa(P)$ is $ \leq c$. It follows that
codimension of $A_P$ is $\leq c$.
\end{proof}


\begin{thebibliography} {99}
\bibitem{Av-1}
 L.~L.~Avramov,
 \emph{L. L. Flat morphisms of complete intersections. (Russian)}
  Dokl. Akad. Nauk SSSR 225 (1975), no. 1, 11--14.

  \bibitem{Av-2}
  \bysame,
   \emph{Homology of local flat extensions and complete intersection defects},
    Math. Ann. 228 (1977), no. 1, 27--37

  \bibitem{AGP}
  \bysame,  V.~N.~Gasharov and I.~V.~Peeva,
   \emph{Complete intersection dimension},
    Inst. Hautes Études Sci. Publ. Math. No. 86 (1997), 67--114 (1998).

    \bibitem{BH}
W.~Bruns and J.~Herzog, \emph{{Cohen-Macaulay rings}},
 revised edition, vol.~39, Cambridge
  studies in advanced mathematics, Cambridge University Press,~Cambridge, 1997.

 \bibitem{GL}
  T.~H.~Gulliksen and  G.~Levin,
\emph{Homology of local rings},
Queen's Paper in Pure and Applied Mathematics, No. 20 Queen's University, Kingston, Ont. 1969.

\bibitem{HJ}
R.~Heitmann and  D.~Jorgensen,
\emph{Are complete intersections complete intersections?}
J. Algebra 371 (2012), 276--299.

\bibitem{M}
S.~MacLane,
\emph{Homology},
 Grundlehren Math. Wiss., vol. 114, Springer, Berlin, 1963.

\end{thebibliography}
\end{document}